\documentclass[12pt]{article}
\usepackage{amsthm}
\usepackage{amsmath}
\usepackage{enumerate}
\usepackage{amssymb}
\usepackage{latexsym}
\usepackage{amsfonts}
\usepackage{color}
\usepackage{mathrsfs}
\usepackage{epsfig}
\newtheorem{theorem}{\bf Theorem}[section]

\newtheorem{definition}[theorem]{\bf Definition}

\newtheorem{remark}[theorem]{\bf Remark}
\newtheorem{lemma}[theorem]{\bf Lemma}

\newsavebox{\savepar}

\begin{document}
	\title{A weighted fractional problem involving a singular nonlinearity and an $L^1$ datum}
	\author{Akasmika Panda$^1$,  Debajyoti Choudhuri$^{2,}$\footnote{Corresponding
			author} \ and Leandro S. Tavares$^3$\\
			\small{Department of Mathematics, National Institute of Technology Rourkela, India$^{1,2}$}\\
			\small{Center of Science and Technology, Federal University of Cariri, Brazil$^{3}$}\\
			\small{Emails: akasmika44@gmail.com$^1$, dc.iit12@gmail.com$^2$, leandro.tavares@ufca.edu.br$^3$}
		}

		\date{}
		\maketitle
		\begin{abstract}
\noindent In this article, we show the existence of a unique entropy solution to the following problem:
\begin{equation}
\begin{split}
(-\Delta)_{p,\alpha}^su&= f(x)h(u)+g(x) ~\text{in}~\Omega,\\
u&>0~\text{in}~\Omega,\\
u&= 0~\text{in}~\mathbb{R}^N\setminus\Omega,\nonumber
\end{split}
\end{equation}
where the domain $\Omega\subset \mathbb{R}^N$ is bounded  and contains the origin, $ \alpha\in[0,\frac{N-ps}{2})$, $s\in (0,1)$, $2-\frac{s}{N}<p<\infty$, $sp<N$, $g\in L^1(\Omega)$, $f\in L^q(\Omega)$ for $q>1$ and $h$ is a general singular function with singularity at 0. Further, the fractional $p$-Laplacian with weight  $\alpha$ is given by 
$$(-\Delta)_{p,\alpha}^su(x)=\text{P. V.}\int_{\mathbb{R}^N}\frac{|u(x)-u(y)|^{p-2}(u(x)-u(y))}{|x-y|^{N+ps}}\frac{dy}{|x|^\alpha|y|^{\alpha}},~\forall x\in \mathbb{R}^N.$$
		\end{abstract}
		\section{Introduction}
We consider the following problem
		\begin{equation}\label{ent}
		\begin{split}
		(-\Delta)_{p,\alpha}^su&= f(x)h(u)+g(x) ~\text{in}~\Omega,\\
		u&>0~\text{in}~\Omega,\\
		u&= 0~\text{in}~\mathbb{R}^N\setminus\Omega,
		\end{split}
		\end{equation}
		where $\Omega$ is a bounded domain in $\mathbb{R}^N$, $s\in (0,1)$, $2-\frac{s}{N}<p<\infty$, $sp<N$,  $0\leq\alpha<\frac{N-ps}{2}$ with  $g\in L^1(\Omega)$ and $f\in L^q(\Omega)$, for some $q>1$ are two positive functions. The singularity $h : \mathbb{R}^+ \rightarrow \mathbb{R}^+$ is a non-increasing and continuous function such that
		\begin{equation}\label{eqn3}
		\underset{v\rightarrow 0^+}{\lim}~ h(v)\in (0,\infty]~ \text{and}~\underset{v\rightarrow \infty}{\lim} h(v):=h(\infty)<\infty.
		\end{equation}
Regarding the  growth conditions on $h$ near zero and infinity it will be considered that there are constants $K_1, K_2, M, N>0$ in such a way that
		\begin{equation}\label{eqq1}
	 h(v)\leq\frac{K_1}{v^\gamma}~\text{if}~v<{M},~\text{for}~	0<\gamma<1,
		\end{equation}
		\begin{equation}\label{eqqq}
	h(v)\leq\frac{K_2}{v^\theta}~\text{if}~v>N,~\text{for}~ \theta>0.
		\end{equation}
Note that in the case \eqref{eqqq} we have $h(\infty)=0$.
		 The weighted fractional $p$-Laplacian is defined as 
		 \begin{equation}\label{frac p laplace}
		 	(-\Delta)_{p,\alpha}^su(x)=\text{P
		 		V}\int_{\mathbb{R}^N}\frac{|u(x)-u(y)|^{p-2}(u(x)-u(y))}{|x-y|^{N+ps}}\frac{dy}{|x|^\alpha|y|^{\alpha}},~\forall x\in \mathbb{R}^N.
		 \end{equation} 
		This operator $(-\Delta)^s_{p,\alpha}$ is the nonlocal version of the local operator
		the operator $-\operatorname{div}(|x|^{-\alpha}|\nabla \cdot|^{p-2}\nabla \cdot)$. The Caffarelli$-$Kohn$-$Nirenberg inequality, derived in \cite{Caffarelli}, is strongly related to the problem
		\begin{equation}\label{CKN}
		-\operatorname{div}(|x|^{-\alpha}|\nabla u|^{p-2}\nabla u)=0,~\text{in}~\mathbb{R}^N
		\end{equation} 
	and each weak supersolution to the above problem $\eqref{CKN}$ satisfies the weak Harnack inequality, for every $\alpha<N-p$. Thus, in this sense, we say that $|x|^{-\alpha}$ is an admissible weight. A nonlocal version of the Caffarelli$-$Kohn$-$Nirenberg inequality, in a bounded domain, has been introduced by Abdellaoui \& Bentifour in \cite{Abdellaoui} with $\alpha=\frac{N-sp}{2}$. For the case $\Omega=\mathbb{R}^N$, the generalization of the classical Caffarelli$-$Kohn$-$Nirenberg inequality is given in \cite{Abdellaoui} for $\alpha<\frac{N-sp}{2}$ and is known as the weighted fractional Sobolev inequality. This condition on $\alpha$ is related in some sense to the admissible weight described in \cite{Heinonen}. The derivation of the weighted fractional Sobolev inequality is based on the following improved Hardy inequality, i.e. for every $u\in C_0^\infty(\Omega)$,
	\begin{align}\label{hardy}
	\int_{\mathbb{R}^{2N}}\frac{|u(x)-u(y)|^{p}}{|x-y|^{N+sp}}dxdy&-\Lambda_{N,s,p}\int_{\mathbb{R^N}}\frac{|u|^p}{|x|^{sp}}dx\nonumber\\&\geq C\int_{\mathbb{R}^{2N}}\frac{|v(x)-v(y)|^{p}}{|x-y|^{N+sp}}\frac{dxdy}{|x|^{\frac{N-sp}{2}}|y|^{\frac{N-sp}{2}}}.
	\end{align}
Here, $\Lambda_{N,s,p}$ is the best Hardy constant and $v(x)=|x|^{\frac{N-sp}{p}}u(x)$. Therefore, the class of operators like $(-\Delta)^s_{p,\alpha}$ appear naturally when one deals with the inequality \eqref{hardy}. We quote  \cite{Abdellaoui, Abdellaoui 2, Frank} for further references.\\
The main goal of this manuscriptis to obtain the existence and uniqueness of an entropy solution to $\eqref{ent}$. The concept of entropy solution is very useful when  technical difficulties arise in proving the uniqueness of the weak solution. Regarding such approach in problems involving an $L^1$ datum we point out   \cite{ Abdellaoui 1, Boccardo, Benilan} and   its references. Such notion was  considered for the first time by  Boccardo et al. \cite{Boccardo} and B\'{e}nilan et al \cite{Benilan}, where the authors considered the following problem with an $L^1$ datum $g$:
\begin{equation}\label{boccardo problem}
\begin{split}
A(u)&= g(x) ~\text{in}~\Omega,\\
u&=0~\text{in}~\Omega,
\end{split}
\end{equation} 
where $u\mapsto A(u)$ is a monotone operator acting on $W_0^{1,p}(\Omega)$ for $1<p<N$. The above mentioned problem $\eqref{boccardo problem}$ is a local sub-case of the problem $\eqref{ent}$ with $s=1, \alpha=0$ and $f=0$. Although there is a good amount of literature pertaining to the case $f\neq0$, the notion of  entropy solution  was not previously applied in the literature  to guarantee the existence and uniqueness of solution. \\
The case of measure data was considered in \cite{Maso}, where the authors have shown the existence of a renormalized solution to the following problem:
\begin{equation}
\begin{split}
-\operatorname{div}(a(x,\nabla u))&= \mu ~\text{in}~\Omega,\\
u&=0~\text{in}~\Omega.
\end{split}
\end{equation} 
 Here, $\mu$ is a bounded Radon measure and $u\mapsto -\operatorname{div}(a (x, \nabla u))$ is a monotone operator acting on $W_0^{1,p}(\Omega)$. Kuusi et al. in \cite{Kuusi}, dealt with a Dirichlet problem involving an $L^1$ datum and the fractional $p$-Laplacian. The authors used the basic and optimal nonlinear Wolff potential estimates to obtain the existence of solution in an appropriate fractional Sobolev space. An important piece of work which is worth mentioning here is due to Abdellaoui et al. in \cite{Abdellaoui 1}. The authors have considered the case $f\equiv0$ of the problem $\eqref{ent}$ and obtained a unique entropy solution using some algebraic inequalities.\\
 An improvement of the problems above was considered by Panda et al. \cite{Panda}, where the authors studied a local singular problem involving a measure datum. The problem $\eqref{ent}$ with $\alpha=0, s=1,p=2$ and with $\alpha=0, s\in (0,1), p=2$  has been analysed by Panda et al. \cite{Panda} and by Ghosh et al. \cite{Ghosh} respectively. The authors  guaranteed the existence of a weak solution as well as a very weak solution to the problem using an approximation argument. The case of purely singular problem, i.e. $\eqref{ent}$ with $g\equiv0$, involving the fractional $p$-Laplace operator was treated in \cite{Canino} for $\gamma>0$. Some equally important works in the literature on these type of problems that involves a singular term and a source term in the form of a measure can be found in \cite{pet0, pet2,pet1, pet3} and the references therein.\\
 In this manuscript, by using adapting the ideas of \cite{Abdellaoui 1, Panda}, we extend the work of Abdellaoui \cite{Abdellaoui 1} by considering a singular function $h$, having a singularity at 0, in the right hand side of the problem. Due to the irregularity near the boundary, singular problems admits solutions in a weak distributional sense, for compactly supported test functions.  See Section $\ref{notations}$ for the notations.
		\begin{theorem}\label{exist weak}
			There exists a positive nontrivial weak solution $u$  to the problem $\eqref{ent}$ in   $X_0^{t,m,\alpha}(\Omega)$, for every $1\leq m<\frac{N(p-1)}{N-s}$ and $0<t<s$, in the sense of Definition $\ref{weak}$. More precisely,  
			\begin{equation}\label{2nd 2nd}
				\int_{Q}\frac{|u(x)-u(y)|^m}{|x-y|^{N+tm}}\frac{dxdy}{|x|^\alpha|y|^\alpha}\leq C,
			\end{equation}
			for every $1\leq m<\frac{N(p-1)}{N-s}$ and every $0<t<s$. Moreover, $u\in \mathcal{T}_0^{s,p,\alpha}(\Omega)$.
		\end{theorem}
	\noindent \noindent The weak solution space do not lie in the natural energy space corresponding to the operator $(-\Delta)_{p,\alpha}^s$, i.e. $X_0^{s,p,\alpha}(\Omega)$, but it has a lower degree of differentiability and integrability. The next theorem proves the uniqueness result.
	\begin{theorem}\label{exist entropy}
		Let $g\in L^1(\Omega)$. Then, the problem $\eqref{ent}$ admits a unique entropy solution in $\mathcal{T}_0^{s,p,\alpha}(\Omega)$ in the sense of Definition $\ref{entropy}$.
	\end{theorem}
\noindent The paper is organized as follows: In Section $\ref{notations}$, we introduce the weighted fractional Sobolev space which is the natural solution space to the given problem $\eqref{ent}$. Further, we also define some useful function spaces and provide some auxiliary results. The existence of a positive weak solution to the problem $\eqref{ent}$ is proved in Section $\ref{section weak}$ by an approximation argument and using some apriori estimates. Section $\ref{section entropy}$ is all about showing the existence and uniqueness of a positive entropy solution to $\eqref{ent}$. In this process, we show the equivalence between the weak solution and  the entropy solution to $\eqref{ent}$. Additionally, we prove the existence of an entropy solution to the problem $\eqref{ent}$ with $h(u)=\frac{1}{u^\gamma}$ and with a general function $g(x,w)=w^r+g(x)$ for $r<p-1$.
\section{Properties of weighted fractional Sobolev space and some preliminary definitions \& results}\label{notations}
We now introduce the weighted fractional Sobolev space which is a natural solution space to the problem $\eqref{ent}$. Let us consider $\Omega$ to be a bounded domain in $\mathbb{R}^N$, $s\in (0,1)$, $1<p<\infty$ and $\alpha\in[0,\frac{N-sp}{2})$. Denote
$$d\nu=\frac{dx}{|x|^{2\alpha}},~~d\mu=\frac{dxdy}{|x-y|^{N+sp}|x|^\alpha|y|^\alpha}.$$ 
The fractional Sobolev space with weight $\alpha$ (refer \cite{Abdellaoui}) is defined by
$$W^{s,p,\alpha}(\mathbb{R}^N)=\left\{u\in L^p(\mathbb{R}^N,d\nu):\iint_{\mathbb{R}^{2N}}|u(x)-u(y)|^pd\mu<\infty\right\}.$$ 
Further, denote $Q=\mathbb{R}^{2N}\setminus(\Omega^c\times\Omega^c)$ with $\Omega^c=\mathbb{R}^N\setminus\Omega$. Then, the space $X^{s,p,\alpha}(\Omega)$ is defined by
\begin{align}
& X^{s,p,\alpha}(\Omega)\nonumber\\&
=\left\{u:\mathbb{R}^N\rightarrow\mathbb{R}~\text{is measurable},u|_{\Omega}\in L^p(\Omega,d\nu) ~\text{and}~\int_Q\frac{|u(x)-u(y)|^p}{|x-y|^{N+sp}}\frac{dxdy}{|x|^\alpha|y|^\alpha}<\infty\right\}\nonumber
\end{align} 
and this space is a Banach space endowed with the norm:
$$\|u\|_{X^{s,p,\alpha}(\Omega)}=\|u\|_{L^p(\Omega,d\nu)}+\left(\int_Q\frac{|u(x)-u(y)|^p}{|x-y|^{N+sp}}\frac{dxdy}{|x|^\alpha|y|^\alpha}\right)^{\frac{1}{p}}.$$
Define a space $X_0^{s,p,\alpha}(\Omega)$ by  
$$X_0^{s,p,\alpha}(\Omega)=\{u\in X^{s,p,\alpha}(\Omega): u=0 \text{ in }\mathbb{R}^N\setminus\Omega\},$$ which is a Banach space when endowed with the following norm: $$\|u\|_{X_0^{s,p,\alpha}(\Omega)}^p=\int_{Q}\frac{|u(x)-u(y)|^p}{|x-y|^{N+sp}}\frac{dxdy}{|x|^\alpha|y|^\alpha}.$$
If $\alpha=0$, denote $W^{s,p,\alpha}(\mathbb{R}^N)=W^{s,p}(\mathbb{R}^N),~X^{s,p,\alpha}(\Omega)=X^{s,p}(\Omega),~X_0^{s,p,\alpha}(\Omega)=X_0^{s,p}(\Omega)$ which are the usual well-known fractional Sobolev spaces. Since the domain $\Omega$ is bounded, the embedding $L^1(\Omega,d\nu)\hookrightarrow L^1(\Omega) $ is continuous and we will use this embedding throughout the article.\\
 The following is the weighted fractional Sobolev inequality which is an improvement of the fractional Sobolev inequality and is obtained by the help of weighted Hardy inequality (see \eqref{hardy}).
\begin{theorem}[\cite{Abdellaoui}, Theorem 1.4, Theorem 1.5]\label{weighted sobo ineq}
	Let $s\in (0,1)$ and $1<p<\frac{N}{s}$. If $\alpha<\frac{N-sp}{2}$, then there exists constant $S=S(s,N,p,\alpha)>0$ such that
	$$\iint_{\mathbb{R}^{2N}}|u(x)-u(y)|^{p}d\mu\geq S\left(\int_{\mathbb{R}^N}\frac{|u(x)|^{p_s^*}}{|x|^{\frac{2\alpha p_s^*}{p}}}\right)^{\frac{p}{p_s^*}}$$
	for every $u\in C_c^\infty(\mathbb{R}^N)$. Here, the constant $p_s^*=\frac{Np}{N-sp}$ is the fractional critical Sobolev exponent.
\end{theorem}
\noindent The next theorem is known as the fractional Caffarelli$-$Kohn$-$Nirenberg inequality in a bounded domain, proved in \cite{Abdellaoui}.
\begin{theorem}
	 If $\Omega$ is a bounded domain in $\mathbb{R}^N$ and $\alpha=\frac{N-sp}{2}$, then for every $q<p$ there exists constant $C=C(\Omega, s, q,N)>0$ such that
	 $$\iint_{\mathbb{R}^{2N}}|u(x)-u(y)|^{p}d\mu\geq
	 C\left(\int_{\mathbb{R}^N}\frac{|u(x)|^{p_{q,s}^*}}{|x|^{\frac{2\alpha p_{q,s}^*}{p}}}\right)^{\frac{p}{p_{q,s}^*}}$$
	 for every $u\in C_c^\infty(\mathbb{R}^N)$, where $p^*_{q,s}=\frac{Np}{N-sq}$.
\end{theorem}
	\noindent The truncation function $T_k:\mathbb{R}\rightarrow \mathbb{R}$, for a fixed $k>0$, is defined  by
	\begin{eqnarray}
	T_k(s)=\max\{-k,\min\{k,s\}\}\nonumber
	\end{eqnarray}
	and define $$G_k(v)=(|v|-k)^+sign(v).$$
	Clearly, $T_k(v)+G_k(v)=v$ for every $v\in \mathbb{R}$.
	\begin{definition}
		We define $\mathcal{T}^{s,p,\alpha}_0(\Omega)$ to be the class of measurable functions $u$ on $\Omega$ such that $T_k(u)\in X^{s,p,\alpha}_0(\Omega)$ for all $k>0$.
	\end{definition}
	\noindent 	The following well-known algebraic inequalities will be frequently used throughout this article.
	\begin{remark}\label{inequalities}
		Let $p\geq 1$,  $\beta>0$ and $a_1,a_2\in [0,\infty]$. Then, there exist constants $c_1,c_2,c_3>0$ such that 
		\begin{enumerate}
			\item $(a_1+a_2)^\beta\leq c_1 a_1^\beta+c_2 a_2^\beta$,
			\item $|a_1-a_2|^{p-2}(a_1-a_2)(a_1^\beta-a_2^\beta)\geq c_3 |a_1^{\frac{\beta+p-1}{p}}-a_2^{\frac{\beta+p-1}{p}}|^p$,
			\item $|a_1+a_2|^{\beta-1}|a_1-a_2|^p\leq c_3 |a_1^{\frac{\beta+p-1}{p}}-a_2^{\frac{\beta+p-1}{p}}|^p$ for $\beta\geq 1$,
			\item for $a_1,a_2\in \mathbb{R}$,  $$|a_1-a_2|^{p-2}(a_1-a_2)(T_k(a_1)-T_k(a_2))\geq |T_k(a_1)-T_k(a_2)|^p,$$
			$$|a_1-a_2|^{p-2}(a_1-a_2)(G_k(a_1)-G_k(a_2))\geq |G_k(a_1)-G_k(a_2)|^p,$$
			\item $\left(a_1^{p-2}a_1-a_2^{p-2}a_2\right)\cdot (a_1-a_2)\geq 2^{2-p}|a_1-a_2|^p$, when $p\geq 2,$
			\item $\left(a_1^{p-2}a_1-a_2^{p-2}a_2\right)\cdot (a_1-a_2)\geq \frac{1}{2} \left(|a_1|+|a_2|\right)^{p-2}|a_1-a_2|^2$, when $p<2$.
		\end{enumerate}
	\end{remark}
\begin{definition}[\cite{Abdellaoui 1}]\label{weighted marcin}
	A measurable function $u:\Omega\rightarrow\mathbb{R}$ is said to be in the weighted Marcinkiewicz space, denoted by $M^q(\Omega,d\nu)$ for $0<q<\infty$, if 
	$$\nu(\{x\in\Omega:|u(x)|>t\})\leq \frac{C}{t^q},\text{ for }~t>0 \text{ and } 0<C<\infty.$$
\end{definition}
\noindent Assume $\Omega\subset\mathbb{R}^N$ to be bounded, then
	\begin{enumerate}
		\item $M^{q_1}(\Omega,d\nu)\subset M^{q_2}(\Omega,d\nu)$, $\forall q_1\geq q_2>0$,
		\item for $1\leq q<\infty$ and $0<\epsilon<q-1$, we have the following continuous embedding
		\begin{equation}\label{continuous}
		L^q(\Omega,d\nu)\subset 	M^q(\Omega,d\nu)\subset 	L^{q-\epsilon}(\Omega,d\nu).
		\end{equation}
	\end{enumerate}
\noindent Abdellaoui \& Bentifour \cite{Abdellaoui} provided a useful comparison principle which is stated in the following lemma.
\begin{lemma}\label{comparison}
	Assume that $\Omega\subset\mathbb{R^N}$ is a bounded domain, $p
	\geq 1$, $H(\cdot,\cdot)$ is a positive continuous function and $\frac{H(x,v)}{v^{p-1}}$ is decreasing for $v>0$. Furthermore, let $v,w\in X^{s,p,\alpha}_0(\Omega)$ satisfy 
	\begin{eqnarray}
	\begin{split}
	&\begin{cases}
	(-\Delta)^s_{p,\alpha}v\geq H(x,v)&\text{in}~ \Omega,\\ 
	(-\Delta)^s_{p,\alpha}w\leq H(x,w)&\text{in}~ \Omega,\nonumber \\
	v,w>0 & \text{in}~ \Omega.\nonumber
	\end{cases}
	\end{split}
	\end{eqnarray}
	Then, $w\leq v$ in $\Omega$.
\end{lemma}
\noindent We now provide different notions of solutions to problem $\eqref{ent}$.
\begin{definition}\label{aux_space}
Let $u$ be a measurable function. Then $u\in \mathcal{T}_{0}^{s,p,\alpha}(\Omega)$ if for every $k>0$, $T_k(u)\in X_0^{s,p,\alpha}(\Omega)$.
\end{definition}
\begin{definition}[Weak solution]\label{weak}
	Let $g\in L^1(\Omega)$. A function $u\in X_0^{t,m,\alpha}(\Omega)$, for $0<t<s$ and $1\leq m<p$, is said to be a positive weak solution to $\eqref{ent}$ if 
	\begin{equation}\label{weak form}
		\int_{Q}|u(x)-u(y)|^{p-2}(u(x)-u(y))(\phi(x)-\phi(y))d\mu=\int_{\Omega}fh(u)\phi+\int_{\Omega}g \phi	
	\end{equation}
for every $\phi\in C_c^\infty(\Omega)$	and for every $\omega\subset \subset\Omega$, there exists $C_\omega>0$ such that $u\geq C_\omega>0$.	
\end{definition}
\noindent	In many cases, it is very tough to guarantee the uniqueness of weak solutions and to overcome this difficulty, the concept of entropy solution comes into picture to prove the uniqueness of solutions. 
\begin{definition}[Entropy solution]\label{entropy}
	Let $g\in L^1(\Omega)$. A function $u\in\mathcal{T}_0^{s,p,\alpha}(\Omega)$ is said to be an entropy solution to $\eqref{ent}$ if it satisfies the following:
	\begin{enumerate}
		\item \begin{equation}\label{condition 1}
		\int_{D_l}|u(x)-u(y)|^{p-1}d\mu\rightarrow 0~~\text{ as } l\rightarrow\infty,		\end{equation}
		where $$D_l=\{(x,y)\in\mathbb{R}^N\times \mathbb{R}^N:l+1\leq\max\{|u(x)|,|u(y)|\}\text{ with } \min\{|u(x)|,|u(y)|\}\leq l\},$$
		\item for every $\varphi\in X^{s,p,\alpha}_0(\Omega)\cap L^\infty(\Omega)$, $k>0$,
		\begin{align}\label{condition 2}
		\int_{Q}|u(x)-u(y)|^{p-2}(u(x)-u(y))&[T_k(u(x)-\varphi(x))-T_k(u(y)-\varphi(y))]d\mu\nonumber\\&\leq\int_{\Omega}fh(u)T_k(u-\varphi)+\int_{\Omega}g T_k(u-\varphi).
		\end{align}
		The solution $u$ is positive if for every $\omega\subset \subset\Omega$, there exists $C_\omega>0$ such that $u\geq C_\omega>0$.
	\end{enumerate}		
\end{definition}
\begin{remark}\label{entropy is weak}
	It is easy to show that every entropy solution of $\eqref{ent}$ is a weak solution of $\eqref{ent}$.
\end{remark}
\section{Existence of weak solution - Proof of Theorem $\ref{exist weak}$}\label{section weak}
In this section, we prove the existence of a positive weak solution to the problem $\eqref{ent}$ in the sense of Definition $\ref{weak}$. The proof is through approximation, since the problem involves an $L^1$ function $g$. Let us consider the following approximating problem:
 \begin{equation}\label{ent approx}
 \begin{split}
 (-\Delta)_{p,\alpha}^su_n&= f_n h_n(u_n+1/n)+g_n ~\text{in}~\Omega,\\
 u_n&>0~\text{in}~\Omega,\\
 u_n&= 0~\text{in}~\mathbb{R}^N\setminus\Omega.
 \end{split}
 \end{equation}
 Here, $f_n=T_n(f)$, $h_n=T_n(h)$ and $\{g_n\}\subset L^\infty(\Omega)$ is an increasing sequence such that $g_n\rightarrow g$ strongly in $L^1(\Omega)$. We say $u_n\in X_0^{s,p.\alpha}(\Omega)$ is a weak solution to $\eqref{ent approx}$, if $u_n$ satisfies
 \begin{equation}\label{weak formula approx}
 \int_{Q}|u_n(x)-u_n(y)|^{p-2}(u_n(x)-u_n(y))(\phi(x)-\phi(y))d\mu=\int_{\Omega}f_nh_n(u_n+1/n)\phi+\int_{\Omega}g_n \phi,
 \end{equation} 
 for every $\phi\in C_c^\infty(\Omega)$.
 \begin{lemma}\label{existence}
Let $h$ satisfies $\eqref{eqq1}$ and $\eqref{eqqq}$ with $0<\gamma,\theta<\infty$. Then, for a fixed $n\in \mathbb{N}$, there exists a unique weak solution  $u_n\in X_0^{s,p,\alpha}(\Omega)\cap L^\infty(\Omega)$ to $\eqref{ent approx}$.
 \begin{proof}
 	The proof follows the ideas of Panda et al.
 	 \cite{Panda} using Schauder's fixed point argument. Let us define a map, $$G:L^p(\Omega,d\nu)\rightarrow X_0^{s,p,\alpha}(\Omega)\subset L^p(\Omega,d\nu)$$
 	such that for any $w\in L^p(\Omega,d\nu)$ we obtain a unique weak solution $v\in X_0^{s,p,\alpha}(\Omega)$ to the following problem by Lax-Milgram theorem.
 	\begin{eqnarray}\label{helping}	
 	(-\Delta)^s_{p,\alpha} v&=& h_n\left(|w|+\frac{1}{n}\right)f_n+g_n ~\text{in}~\Omega,\nonumber\\
 	v&=& 0 ~\text{on}~ \partial\Omega.\label{schauder} 	
 	\end{eqnarray}
 	Further, $C_c^{\infty}(\Omega)$ is dense in $X_0^{s,p,\alpha}(\Omega)$ with respect to the norm topology of $X_0^{s,p,\alpha}(\Omega)$. Hence, we are allowed to choose $v$ as a test function in the weak formulation of $\eqref{helping}$. Thus, we get
 	\begin{align}\label{equation}
 	\int_{Q}|v(x)-v(y)|^{p}d\mu
 	& = {\int _\Omega h_n\left(|w|+\frac{1}{n}\right)f_n v} +\int _\Omega g_n v\nonumber
 	\\& \leq{ K_1 \int _{(|w|+\frac{1}{n}<M)}\frac{f_n v} {(|w|+\frac{1}{n})^\gamma}} + \max_{[M,N]} \{h_n\} \int_{(M\leq (|w|+\frac{1}{n})\leq N)} f_n v \nonumber\\& ~~~~~~+ K_2 \int _{(|w|+\frac{1}{n} > N)}\frac{f_n v} {(|w|+\frac{1}{n})^\theta} + C(n) \int_\Omega|v|\nonumber\\
 	& \leq K_1 n^{1+\gamma}\int _{(|w|+\frac{1}{n}< M)}|v| + n \max_{[M,N]} \{h_n\} \int_{(M\leq (|w|+\frac{1}{n})\leq N)}  |v| \nonumber
 	\\& ~~~~~~ + K_2n^{1+\theta} \int _{(|w|+\frac{1}{n} > N)} |v| + C(n) \int_\Omega|v|\nonumber
 	\\& \leq C^\prime (n,\gamma,\theta) \int _\Omega |v|d\nu= C^\prime (n,\gamma,\theta) \|v\|_{L^1(\Omega,d\nu)}.\nonumber
 	\end{align}
 Now by applying the weighted Sobolev embedding theorem, i.e. Theorem $\ref{weighted sobo ineq}$, there exists a constant $C(n,\gamma,\theta)>0$ such that
 	\begin{equation}\label{eq}
 	\|v\|_{X_0^{s,p,\alpha}(\Omega)}\leq C(n,\gamma,\theta),
 	\end{equation}
 	where $C(n,\gamma,\theta)$ is independent of $w$. We will now show the continuity and compactness of the operator $G$ to apply the Schauder fixed point theorem.\\
 	\textbf{Claim 1:} The map $G$ is continuous over $X_0^{s,p,\alpha}(\Omega)$.\\
 Let us consider a sequence $\{w_k\}\subset X_0^{s,p,\alpha}(\Omega)$ that converges strongly to $w$ with respect to the $X_0^{s,p,\alpha}$-norm. Thus, by the uniqueness of the weak solution, we denote $v_k=G(w_k)$, $v=G(w)$ and let $\bar{v_k}=v_k-v$. According to the Remark $\ref{inequalities}$-(5) and (6),
 \begin{align}
 \text{if } 1<p<2, \int_{Q}&\left(|v_k(x)-v_k(y)|+|v(x)-v(y)|\right)^{p-2}|\bar{v_k}(x)-\bar{v_k}(y)|^2 d\mu\nonumber\\&\leq 2\int_{\Omega}\left(h_n\left(|w_k| +\frac{1}{n}\right)f_n-h_n\left(|w| +\frac{1}{n}\right)f_n\right) \bar{v_k}
  \end{align}
  and
 \begin{align}
 \text{if } p\geq2, \int_{Q}|\bar{v_k}(x)-\bar{v_k}(y)|^p d\mu\leq 2^{p-2}\int_{\Omega}\left(h_n\left(|w_k| +\frac{1}{n}\right)f_n-h_n\left(|w| +\frac{1}{n}\right)f_n\right) \bar{v_k}.
 \end{align}
 Hence, by Theorem $\ref{weighted sobo ineq}$ and H\"{o}lder's inequality we establish the following for the case of $p\geq 2$.
 \begin{align}
 \|\bar{v_k}\|^p_{X_0^{s,p,\alpha}(\Omega)}&\leq 2^{p-2}\left\| h_n\left(|w_k| +\frac{1}{n}\right)f_n-h_n\left(|w| +\frac{1}{n}\right)f_n\right\|_{L^{(p_s^*)^\prime}(\Omega)}\|\bar{v_k}\|_{L^{p_s^*}(\Omega)}\nonumber\\&\leq C_1(N,p,\alpha,\Omega)\left\| h_n\left(|w_k| +\frac{1}{n}\right)f_n-h_n\left(|w| +\frac{1}{n}\right)f_n\right\|_{L^{(p_s^*)^\prime}(\Omega)}\left\|\frac{\bar{v_k}}{|x|^{2\alpha/p}}\right\|_{L^{p_s^*}(\Omega)}\nonumber\\&\leq C_2 \left\| h_n\left(|w_k| +\frac{1}{n}\right)f_n-h_n\left(|w| +\frac{1}{n}\right)f_n\right\|_{L^{(p_s^*)^\prime}(\Omega)}\left\|\bar{v_k}\right\|_{X_0^{s,p,\alpha}(\Omega)},\nonumber
 \end{align}
 where $(p_s^*)'$ is the H\"{o}lder conjugate of $p_s^*$. This implies 
 \begin{equation}
\|\bar{v_k}\|_{X_0^{s,p,\alpha}(\Omega)}\leq C \left\| h_n\left(|w_k| +\frac{1}{n}\right)f_n-h_n\left(|w| +\frac{1}{n}\right)f_n\right\|_{L^{(p_s^*)^\prime}(\Omega)}^{1/{p-1}}. \nonumber
 \end{equation}
 Now on applying the dominated convergence theorem we prove that $\lim\limits_{k\rightarrow \infty}\|\bar{v_k}\|_{X_0^{s,p,\alpha}(\Omega)}=0$. This gives $\lim\limits_{k\rightarrow \infty}\|v_k-v\|_{X_0^{s,p,\alpha}(\Omega)}=0$.	In this similar manner, for the case $1<p<2$, we can prove the strong convergence  $v_k\rightarrow v$ in $X_0^{s,p,\alpha}(\Omega)$. Thus, $G$ is continuous from $X_0^{s,p,\alpha}(\Omega)$ to $X_0^{s,p,\alpha}(\Omega)$.\\
 	\textbf{Claim 2:}  $G(X_0^{s,p,\alpha}(\Omega))$ is relatively compact in $X_0^{s,p,\alpha}(\Omega)$.\\
 Let $\{w_k\}$ is a bounded sequence in $X_0^{s,p,\alpha}(\Omega)$. Then, there exists $w\in X_0^{s,p,\alpha}(\Omega)$ and a subsequence of $\{w_k\}$ (still denoted as $\{w_k\}$) such that $w_k\rightarrow w$ weakly in $X_0^{s,p,\alpha}(\Omega)$. We have already proved in $\eqref{eq}$ that
 	\begin{align}
 		\|G(w_k)\|_{X_0^{s,p,\alpha}(\Omega)}\leq C,\nonumber
 	\end{align}
 	for $C>0$ independent of $k$. Thus, there exists $v\in X_0^{s,p,\alpha}(\Omega)$ such that $G(w_k)\rightarrow v$ weakly in $X_0^{s,p,\alpha}(\Omega)$. Now by passing the limit $k\rightarrow\infty$ in the weak formulation and using some basic computation we prove that $G(w)=v$. Following the similar lines used in the proof of Claim 1, we finally show that $\lim\limits_{k\rightarrow \infty}\|G(w_k)-G(w)\|_{X_0^{s,p,\alpha}(\Omega)}=0$. Hence the claim.\\
 	Thus, on using the Schauder fixed point theorem to $G$, we obtain a fixed point $u_n\in X_0^{s,p,\alpha}(\Omega)$ that is also a weak solution to the problem $\eqref{ent approx}$ in $X_0^{s,p,\alpha}(\Omega)$. Further, by the strong maximum principle [Lemma 2.3, \cite{Mosconi}], $u_n>0$ in $\Omega$. \\
 	\textbf{Claim 3:} Uniqueness of weak solution.\\
 	To prove this claim, suppose the problem $\eqref{ent approx}$ admits two different weak solutions $u_n$ and $v_n$. Let us take $\phi=(u_n-v_n)^+$ as a test function in the weak formulation $\eqref{weak formula approx}$. Thus, for $p\geq 2$, with the consideration of Lemma 9 in \cite{Lindgren}, we obtain the following.
 	\begin{align}
 0 &\leq \int_{Q} |(u_n(x)-v_n(x))^+-(u_n(y)-v_n(y))^+|^p d\mu\nonumber\\& \leq \int_{Q}\left[|u_n(x)-u_n(y)|^{p-2}(u_n(x)-u_n(y))-|v_n(x)-v_n(y)|^{p-2}(v_n(x)-v_n(y))\right]\nonumber\\&\nonumber~~~~~~~~~~~\times[(u_n(x)-v_n(x))^+-(u_n(y)-v_n(y))^+]d\mu\nonumber\\&=\int_{\Omega}\left(f_nh_n(u_n+1/n)-f_nh_n(v_n+1/n)\right)(u_n-v_n)^+\nonumber\\&\leq 0.\nonumber
 	\end{align}
 	This implies $(u_n-v_n)^+=0$ a.e in $\Omega$, since $u_n=v_n=0$ in $\mathbb{R}^N\setminus\Omega$ and thus  $u_n\leq v_n$ a.e. in $\Omega$. The same proof holds for $1<p<2$. Proceeding similarly with $\phi=(v_n-u_n)^+$ as a test function, we can show that $u_n\geq v_n$ a.e. in $\Omega$. This proves the claim.\\
 	Since the right hand side of $\eqref{ent approx}$ belongs to $L^\infty(\Omega)$, by virtue of Lemma 4.3 of \cite{Abdellaoui 1}, we conclude that $u_n\in L^\infty(\Omega)$.			 
 \end{proof}	
 \end{lemma}
\begin{lemma}\label{a priori}
	Let $u_n$ be a unique positive weak solution to $\eqref{ent approx}$. Then, 
	\begin{enumerate}
		\item the sequence $\{u_n\}$ is an increasing sequence w.r.t $n$ and for every set $\omega\subset\subset\Omega$ there exists constant $C_\omega>0$, independent of $n$, such that
		\begin{equation}\label{positive}
		u_n\geq C_\omega > 0,\text{ for every } n\in \mathbb{N}.	
		\end{equation}
		\item $h(u_n)f\in L^1(\Omega)$.
	\end{enumerate}
\begin{proof}
Let us consider the problems satisfied by $u_n$ and $u_{n+1}$. Then, subtracting these two problems and taking the test function $(u_n-u_{n+1})^+$ in its weak formulation we get
	\begin{align}
	0 & \leq \int_{Q} |(u_n(x)-u_{n+1}(y))^+-(u_n(x)-u_{n+1}(y))^+|^p d\mu\nonumber\\&\leq \int_{Q}\left(I_p(u_n)-I_p(u_{n+1}\right)((u_n(x)-u_{n+1}(x))^+-(u_n(y)-u_{n+1}(y))^+)d\mu\nonumber\\&=\int_{\Omega}\left(f_nh_n(u_n+1/n)-f_{n+1}h_{n+1}(u_{n+1}+1/\{n+1\})\right)(u_n-u_{n+1})^++\int_{\Omega}(g_n-g_{n+1})(u_n-u_{n+1})^+\nonumber\\&\leq \int_{\Omega}\left(f_{n+1}h_{n+1}(u_n+1/\{n+1\})-f_{n+1}h_{n+1}(u_{n+1}+1/\{n+1\})\right)(u_n-u_{n+1})^+\nonumber\\&\leq 0,\nonumber
	\end{align}
	where $I_p(v)=|v(x)-v(y)|^{p-2}(v(x)-v(y))$. This implies $(u_n-u_{n+1})^+=0$ a.e. in $\Omega$ and hence $u_n\leq u_{n+1}$ a.e. in $\Omega$.\\
	Further, from the definition of the problem \eqref{ent approx}, $u_1>0$ in $\Omega$ and belongs to $L^{\infty}(\Omega)$, i.e. there exists $M>0$ such that $\|u_1\|_{L^\infty(\Omega)}\leq M$. Since $\{u_n\}$ is an increasing sequence, $u_n$ verifies $\eqref{positive}$ for all $n\geq1$. This proves $(1)$.\\
	Consider the following nonlinear eigenvalue problem:
	\begin{equation}
	\begin{split}
	(-\Delta)_{p,\alpha}^sv&= \lambda |v|^{p-2}v ~\text{in}~\Omega,\\
	v&= 0~\text{in}~\mathbb{R}^N\setminus\Omega.
	\end{split}
	\end{equation}
Moreover, from the weighted fractional Sobolev embedding, we have the Rayleigh quotient as follows.
$$\lambda_1=\underset{v\in X_0^{s,p,\alpha}(\Omega),v\neq0}{\inf}\left\{\frac{\|v\|_{X_0^{s,p,\alpha}(\Omega)}^p}{\|v\|^p_{L^p(\Omega)}}\right\}$$
and $\lambda_1\in(0,\infty)$. Let $\phi_1$ be the eigenfunction corresponding to the eigenvalue $\lambda_1$. Then, it is not difficult to prove that either $\phi_1>0$ a.e. in $\Omega$ or $\phi_1<0$ a.e. in $\Omega$, refer Proposition 2.3 of \cite{Iannizzotto}. Indeed, if $\psi=\phi_1^-$ is chosen as a test function to the following weak formulation
$$\int_{Q}|\phi_1(x)-\phi_1(y)|^{p-2}(\phi_1(x)-\phi_1(y))(\psi(x)-\psi(y))d\mu=\lambda_1\int_\Omega|\phi_1|^{p-2}\phi_1\psi,$$
we get $\phi_1^-\equiv0$ a.e. in $\Omega$. So, $\phi_1>0$ a.e. in $\Omega$. We already know that $u_1\in X_0^{s,p,\alpha}(\Omega)$ is a weak solution to $\eqref{ent approx}$ with $n=1$ and $h$ is a nonincreasing function. Thus, we may choose $c>0$ sufficiently small such that 
\begin{align}
(-\Delta)_{p,\alpha}^s(c\phi_1)-f_1h_1(c\phi_1+1)-g_1&<0\nonumber\\&=(-\Delta)_{p,\alpha}^su_1-f_1h_1(u_1+1)-g_1,~\text{in }\Omega.\nonumber
\end{align}
By using the comparison principle, stated in Lemma $\ref{comparison}$, we observe $c\phi_1\leq u_1$ in $\Omega$. Following the arguments as in Lazer and McKenna \cite{lazer}, for $\gamma<1$, we get $h(\phi_1)\in L^1(\Omega)$. This implies $h(u_1)f\leq h(c\phi_1)f\in L^1(\Omega)$. Therefore, $h(u_n)f\in L^1(\Omega)$ for each $n$ and this concludes the proof of $(2)$. 
\end{proof}
\end{lemma}
\noindent In the proceeding lemma we find some apriori estimates so that we can pass the limit $n\rightarrow \infty$ in $\eqref{weak formula approx}$.
\begin{lemma}\label{bounded}
	Let $h$ satisfies the growth conditions given in 
	$\eqref{eqq1}$, $\eqref{eqqq}$ for $0<\gamma<1$, $\theta>0$ and let $u_n\in X_0^{s,p,\alpha}(\Omega)$ be a positive weak solution to $\eqref{ent approx}$. Then, the sequence $\{u_n\}$ is bounded in $L^{m_1}(\Omega,d\nu)$ for every $m_1<\frac{N(p-1)}{N-sp}$. Furthermore, 
		\begin{equation}\label{1st 1st}
		\int_{Q}\frac{|u_n(x)-u_n(y)|^m}{|x-y|^{N+tm}}\frac{dxdy}{|x|^\alpha|y|^\alpha}\leq C,
		\end{equation} 
for every $1\leq m<\frac{N(p-1)}{N-s}$ and $0<t<s$.
	\begin{proof}
 Let $u_n\in  X_0^{s,p,\alpha}(\Omega)$ be a unique positive weak solution to $\eqref{ent approx}$. Then, for any $k\geq1$, using $\phi=T_k(u_n)$ as a test function in $\eqref{weak formula approx}$, we have
		\begin{align}\label{entropy1}
	\int_{Q}|u_n(x)-u_n(y)|^{p-2}&(u_n(x)-u_n(y))(T_k(u_n(x)-T_k(u_n(y))d\mu \nonumber\\& =\int_{\Omega}f_nh_n(u_n+1/n)T_k(u_n)+\int_{\Omega}g_nT_k(u_n)\nonumber\\&\leq{ K_1 \int _{(u_n +\frac{1}{n}< M)}\frac{f_n T_k(u_n)} {(u_n +\frac{1}{n})^\gamma}} + \max_{[M,N]} \{h(s)\} \int_{(M\leq u_n+\frac{1}{n}\leq N)} f_n T_k(u_n) \nonumber \\& ~~~~+ K_2 \int _{(u_n+\frac{1}{n} > N)}\frac{f_n T_k(u_n)} {(u_n+\frac{1}{n})^\theta}+k\int_{\Omega}g_n\nonumber\\&\leq  K_1M^{1-\gamma} \int _{(u_n+\frac{1}{n}< M)}f + k \max_{[M,N]} \{h(s)\} \int_{(M\leq u_n+\frac{1}{n}\leq N)}f\nonumber\\&~~~~+ \frac{K_2 k}{N^{\theta}} \int _{(u_n+\frac{1}{n} > N)} f+ C_1k\nonumber \\&\leq C_1k
		\end{align}
On using Theorem $\ref{weighted sobo ineq}$, Remark $\ref{inequalities}$ and the above inequality $\eqref{entropy1}$, we have
\begin{align}\label{truncation bound}
S\left(\int_{\Omega}\frac{|T_k(u_n)(x)|^{p_s^*}}{|x|^{\frac{2\alpha p_s^*}{p}}}dx\right)^{\frac{p}{p_s^*}}&\leq \int_{Q}|T_k(u_n)(x)-T_k(u_n)(y)|^p d\mu\nonumber\\&\leq 	\int_{Q}|u_n(x)-u_n(y)|^{p-2}(u_n(x)-u_n(y))(T_k(u_n(x)-T_k(u_n(y))d\mu\nonumber\\& \leq C_1k,~\forall k>0.
\end{align}
Hence,
\begin{align}
\nu(\{x\in\Omega:|u_n(x)|\geq k\})&\leq \nu(\{x\in\Omega:|T_k(u_n(x))|=k\})\nonumber\\&\leq C_2\int_{\Omega}\frac{|T_k(u_n)(x)|^{p_s^*}}{k^{p_s^*}|x|^{\frac{2\alpha p_s^*}{p}}}dx \nonumber\\&\leq \frac{C}{k^{\frac{N(p-1)}{N-sp}}}.
\end{align}
Therefore, by Definition $\ref{weighted marcin}$,  $\{u_n\}$ is bounded in $M^{\frac{N(p-1)}{N-sp}}(\Omega,d\nu)$ and hence bounded in $L^{m_1}(\Omega,d\nu)$ for every $m_1<\frac{N(p-1)}{N-sp}$. Consequently, the sequence $\{|u_n|^{p-2}u_n\}$ is bounded in $L^{m_2}(\Omega,d\nu)$ for every $m_2<\frac{N}{N-sp}$.\\
Let us choose $m<\frac{N(p-1)}{N-s}$ and $t<s$ very close to $s$ such that $\frac{mp(s-t)}{p-m}<\alpha$. Denote $$w_n(x)=T_1(u_n(x))-\frac{1}{(u_n(x)+1)^\beta},\text{ for some }\beta>0 \text{ which will be chosen later}.$$
On using $w_n$ as a test function in $\eqref{weak formula approx}$, we establish the following.
\begin{align}\label{3.2 refer}
&\int_{Q}|u_n(x)-u_n(y)|^{p-2}(u_n(x)-u_n(y))(T_1(u_n(x))-T_1(u_n(y)))d\mu\nonumber\\&~~~~+\int_{Q}|u_n(x)-u_n(y)|^{p-2}(u_n(x)-u_n(y))\left(\frac{1}{(u_n(y)+1)^\beta}-\frac{1}{(u_n(x)+1)^\beta}\right)\nonumber\\&=\int_{\Omega}f_nh_n(u_n+1/n)\left(T_1(u_n(x))-\frac{1}{(u_n(x)+1)^\beta}\right)+\int_{\Omega}g_n\left(T_1(u_n(x))-\frac{1}{(u_n(x)+1)^\beta}\right)\nonumber\\&\leq{ K_1 \int _{(u_n +\frac{1}{n}< M)}\frac{f_n T_1(u_n)} {(u_n +\frac{1}{n})^\gamma}} + \max_{[M,N]} \{h(s)\} \int_{(M\leq u_n+\frac{1}{n}\leq N)} f_n T_1(u_n) + K_2 \int _{(u_n+\frac{1}{n} > N)}\frac{f_n T_1(u_n)} {(u_n+\frac{1}{n})^\theta}+C_1\nonumber\\&\leq  K_1M^{1-\gamma} \int _{(u_n+\frac{1}{n}< M)}f +  \max_{[M,N]} \{h(s)\} \int_{(M\leq u_n+\frac{1}{n}\leq N)}f+ \frac{K_2}{N^{\theta}} \int _{(u_n+\frac{1}{n} > N)} f+ C_1\nonumber \\&\leq C_2.
\end{align}
	On the other hand, the use of Remark $\ref{inequalities}$ implies that
	\begin{align}\label{3.2 refer 1}
	& \int_{Q}\frac{|(u_n+1)^{\frac{p+\beta-1}{p}}(x)-(u_n+1)^{\frac{p+\beta-1}{p}}(y)|^p}{(u_n(x)+1)^\beta(u_n(y)+1)^\beta}d\mu\nonumber\\&\leq\int_{Q}|u_n(x)-u_n(y)|^{p-2}(u_n(x)-u_n(y))\frac{(u_n(x)+1)^\beta-(u_n(y)+1)^\beta}{(u_n(x)+1)^\beta(u_n(y)+1)^\beta}d\mu\nonumber\\&\leq\int_{Q}|T_1(u_n(x))-T_1(u_n(y))|^{p}d\mu\nonumber\\&~~~~+\int_{Q}|u_n(x)-u_n(y)|^{p-2}(u_n(x)-u_n(y))\left(\frac{1}{(u_n(y)+1)^\beta}-\frac{1}{(u_n(x)+1)^\beta}\right)\nonumber\\&\leq\int_{Q}|u_n(x)-u_n(y)|^{p-2}(u_n(x)-u_n(y))(T_1(u_n(x))-T_1(u_n(y)))d\mu\nonumber\\&~~~~+\int_{Q}|u_n(x)-u_n(y)|^{p-2}(u_n(x)-u_n(y))\left(\frac{1}{(u_n(y)+1)^\beta}-\frac{1}{(u_n(x)+1)^\beta}\right).
	\end{align}
	Thus, the inequalities $\eqref{3.2 refer}$ and $\eqref{3.2 refer 1}$ together yields
	\begin{equation}\label{1st 1}
	\int_{Q}\frac{|(u_n+1)^{\frac{p+\beta-1}{p}}(x)-(u_n+1)^{\frac{p+\beta-1}{p}}(y)|^p}{(u_n(x)+1)^\beta(u_n(y)+1)^\beta}d\mu\leq C.
	\end{equation}
	Abdellaoui et al. [Lemma 3.2, \cite{Abdellaoui 1}] established the following inequality for certain range of $\beta>0$.
	\begin{equation}\label{1st 2}
	\int_{\Omega}\int_{\Omega}\frac{|u_n(x)-u_n(y)|^m}{|x-y|^{N+tm}}\frac{dxdy}{|x|^\alpha|y|^\alpha}\leq C_4\left(\int_{Q}\frac{|(u_n+1)^{\frac{p+\beta-1}{p}}(x)-(u_n+1)^{\frac{p+\beta-1}{p}}(y)|^p}{(u_n(x)+1)^\beta(u_n(y)+1)^\beta}d\mu\right)^{\frac{m}{p}},
	\end{equation}
	for  $m<\frac{N(p-1)}{N-s}$ and $0<t<s$. Then, by combining $\eqref{1st 1}$ and $\eqref{1st 2}$, the following is obtained.  
	\begin{equation}\label{1st}
	\int_{\Omega}\int_{\Omega}\frac{|u_n(x)-u_n(y)|^m}{|x-y|^{N+tm}}\frac{dxdy}{|x|^\alpha|y|^\alpha}\leq C.
	\end{equation}
We can always choose a $R>0$ such that $|x-y|\sim |y|$ for $y\in \mathbb{R^N}\setminus B_R$, where $B_R=\{x\in\mathbb{R}^N: |x|<R\}$. 	Assume that $(x,y)\in \Omega\times (B_R\setminus
\Omega)$, then $$\underset{(x,y)\in \Omega\times (B_R\setminus
	\Omega)}{\sup}\left\{\frac{1}{|x-y|^{N+mt}}\right\}\leq C<\infty.$$
Since $\{u_n\}$ is bounded in $L^{m_1}(\Omega,d\nu)$ for every $m_1<\frac{N(p-1)}{N-sp}$, the following is obtained for $m<\frac{N(p-1)}{N-s}$.
\begin{align}\label{ball R}
\int_{\Omega}\int_{B_R\setminus
	\Omega}\frac{|u_n(x)-u_n(y)|^m}{|x-y|^{N+mt}}\frac{dxdy}{|x|^\alpha|y|^\alpha}&=\int_{\Omega}\int_{B_R\setminus
	\Omega}\frac{|u_n(x)|^m}{|x-y|^{N+mt}}\frac{dxdy}{|x|^\alpha|y|^\alpha}\nonumber\\&\leq C \int_{\Omega}\int_{B_R\setminus
	\Omega}|u_n(x)|^m\frac{dxdy}{|x|^\alpha|y|^\alpha}	\nonumber\\&\leq C_3
\end{align} 
and similarly 
\begin{align}\label{outside ball R}
\int_{\Omega}\int_{\mathbb{R^N}\setminus B_R}\frac{|u_n(x)-u_n(y)|^m}{|x-y|^{N+mt}}\frac{dxdy}{|x|^\alpha|y|^\alpha}&=\int_{\Omega}\int_{\mathbb{R^N}\setminus B_R}\frac{|u_n(x)|^m}{|x|^{\alpha}}\frac{dxdy}{|y|^{N+mt+\alpha}}\nonumber\\&\leq C_4 \int_{\mathbb{R^N}\setminus B_R}\frac{dy}{|y|^{N+mt+\alpha}}	\nonumber\\&\leq C_5.
\end{align} 
On combining $\eqref{ball R}$ and $\eqref{outside ball R}$ we obtain 
\begin{equation}\label{2nd}
\int_{\Omega}\int_{\mathbb{R^N}\setminus \Omega}\frac{|u_n(x)-u_n(y)|^m}{|x-y|^{N+mt}}\frac{dxdy}{|x|^\alpha|y|^\alpha}.
\end{equation}
In a similar way we can prove 
\begin{equation}\label{3rd}
\int_{\mathbb{R^N}\setminus \Omega}\int_{ \Omega}\frac{|u_n(x)-u_n(y)|^m}{|x-y|^{N+mt}}\frac{dxdy}{|x|^\alpha|y|^\alpha}.
\end{equation}
	With the consideration of $\eqref{1st}$, $\eqref{2nd}$ and $\eqref{3rd}$, the claim given in $\eqref{1st 1st}$ is proved which concludes the proof.
	\end{proof}
\end{lemma}
	\noindent We now prove our first main result, i.e. the existence of positive weak solution to $\eqref{ent}$ stated in Theorem $\ref{exist weak}$.
	\begin{proof}[Proof of Theorem $\ref{exist weak}$]
	 According to Lemma $\ref{bounded}$, the sequence $\{T_k(u_n)\}$ is bounded in $X_0^{s,p,\alpha}(\Omega)$ and $\{|u_n|^{p-2}u_n\}$ is bounded in $L^{m_2}(\Omega,d\nu)$,  for every $ m_2<\frac{N}{N-sp}$. Thus, there exists a measurable function $u$ such that, upto a subsequential level, $T_k(u_n)\rightarrow T_k(u)$ weakly in $X_0^{s,p,\alpha}(\Omega)$, $|u|^{p-2}u\in L^{m_2}(\Omega,d\nu)$ and $T_k(u)\in X_0^{s,p,\alpha}(\Omega)$. Hence, $u_n\rightarrow u$ a.e. in $\Omega$ and $u\equiv 0$ in $\mathbb{R}^N\setminus\Omega$. \\
	Let us fix the following notations: $$\Phi(x,y)=\phi(x)-\phi(y),~ I_n(x,y)=|u_n(x)-u_n(y)|^{p-2}(u_n(x)-u_n(y)),$$ $$I(x,y)=|u(x)-u(y)|^{p-2}(u(x)-u(y)),~d\mu=\frac{1}{|x-y|^{N+sp}}\frac{dx dy}{|x|^\alpha|y|^\alpha}.$$
	Then, from $\eqref{weak formula approx}$ we have
	\begin{align}
	\int_{Q} I_n(x,y)\Phi(x,y)d\mu&=\int_{\Omega}f_n h_n(u_n+1/{n})\phi+\int_{\Omega}g_n \phi,~\forall \phi\in C_c^\infty(\Omega).\nonumber
	\end{align}
	On rewriting the above equation, we obtain
	\begin{align}\label{entropy2}
	\int_{Q} I(x,y)\Phi(x,y)d\mu+\int_{Q} (I_n(x,y)-I(x,y))\Phi(x,y)d\mu= \int_{\Omega}f_n h_n(u_n+1/{n})\phi+\int_{\Omega}g_n \phi.
	\end{align}
	Since $g_n\rightarrow g$ strongly in $L^1(\Omega)$, we have $$\lim\limits_{n\rightarrow\infty}\int_{\Omega}g_n\phi=\int_{\Omega}g \phi,~\forall \phi\in C_c^\infty(\Omega).$$
	Now with the help of the dominated convergence theorem and $\eqref{positive}$, we are able to pass the limit $n\rightarrow \infty$ in the following integral.
	$$\lim\limits_{n\rightarrow\infty}\int_{\Omega}f_n h_n(u_n+1/{n})\phi=\int_{\Omega}h(u)f\phi.$$
We express
	\begin{align}
	\int_{Q} (I_n(x,y)-I(x,y))\Phi(x,y)d\mu&=\int_{\Omega\times\Omega} (I_n(x,y)-I(x,y))\Phi(x,y)d\mu\nonumber\\&~~+\int_{(\mathbb{R}^{N}\setminus\Omega)\times\Omega} (I_n(x,y)-I(x,y))\Phi(x,y)d\mu\nonumber\\&~~+\int_{\Omega\times(\mathbb{R}^{N}\setminus\Omega)} (I_n(x,y)-I(x,y))\Phi(x,y)d\mu\nonumber\\&=T_{1,n}+T_{2,n}+T_{3,n}.\nonumber
	\end{align}
	Clearly, $I_n\rightarrow I$ a.e. in $\mathbb{R}^N\times\mathbb{R}^N$. By the Vitali's lemma and Lemma $\ref{a priori}$, $I_n\rightarrow I$ strongly in $L^1(\Omega\times \Omega, d\mu)$. Therefore, $T_{1,n}\rightarrow 0$ as $n\rightarrow\infty$. \\
	Consider $(x,y)\in (B_R\setminus\Omega)\times\Omega$, then $$\underset{(x,y)\in (B_R\setminus\Omega)\times\Omega}{\sup}\left\{\frac{1}{|x-y|^{N+sp}}\right\}\leq C<\infty$$
	and $$\left|\frac{(I_n(x,y)-I(x,y))\Phi(x,y)}{|x-y|^{N+sp}|x|^\alpha|y|^\alpha}\right|\leq C \frac{(|u_n(y)|^{p-1}+|u(y)|^{p-1})|\phi(y)|}{|x|^\alpha|y|^\alpha}=C J_n(x,y).$$
	Denote $J(x,y)=\frac{2|u(y)|^{p-1}|\phi(y)|}{|x|^\alpha|y|^\alpha}$. Then, $J_n\rightarrow J$ strongly in $L^1((B_R\setminus\Omega)\times\Omega)$.
	 Hence, by the dominated convergence theorem, $T_{2,n}\rightarrow 0$ as $n\rightarrow\infty$ and similarly we can also prove that $T_{3,n}\rightarrow 0$ as $n\rightarrow \infty$. Hence, on passing the limit $n\rightarrow \infty$ in $\eqref{entropy2}$, we have
	\begin{align}\label{3rd 3rd}
	\int_{Q} I(x,y)\Phi(x,y)d\mu= \int_{\Omega}h(u)f\phi+\int_{\Omega} g \phi,~\forall \phi\in C_c^\infty(\Omega).
	\end{align}
Further, $\eqref{2nd 2nd}$ follows from $\eqref{1st 1st}$ and the Fatou's lemma. Finally, with the consideration of Lemma $\ref{a priori}$ and $\eqref{3rd 3rd}$, it is proved that $u$ is a weak solution to $\eqref{ent}$ in the sense of Definition $\ref{weak}$. Hence the proof.
	\end{proof}
	\begin{remark}
	If we assume $g$ to be a bounded Radon measure, then the above existence result, i.e. Theorem $\ref{exist weak}$, holds.
	\end{remark}
	\section{Existence of entropy solution - Proof of Theorem $\ref{exist entropy}$}\label{section entropy}
  The problem $\eqref{ent}$ admits a positive weak solution by Theorem $\ref{exist weak}$. According to Remark $\ref{entropy is weak}$, every entropy solution of $\eqref{ent}$ is also a weak solution to $\eqref{ent}$, i.e. the entropy solution satisfies $\eqref{weak form}$. The following lemma is a consequence of Lemma 3.8 of \cite{Abdellaoui 1}, which says that every entropy solution to $\eqref{ent}$ satisfies $\eqref{weak form}$ for a larger class of test function space.  
  \begin{lemma}[Lemma 3.8, \cite{Abdellaoui 1}]\label{another weak form}
  	Suppose $u$ is a positive entropy solution to $\eqref{ent}$. Then, $u$ satisfies 
  	$$\int_{Q}|u(x)-u(y)|^{p-2}(u(x)-u(y))(\psi(x)-\psi(y))d\mu=\int_{\Omega}fh(u)\psi+\int_{\Omega}g \psi,$$ 
  	for every $\psi\in X_0^{s,p,\alpha}(\Omega)\cap L^\infty(\Omega)$ s.t. $\psi\equiv0$ in the set $\{u>l\}$ for some $l>0$.
  \end{lemma}
\noindent Let $u_n$ be a  weak solution to $\eqref{ent approx}$. Hence, the sequence $\{u_n\}$ is an increasing sequence w.r.t $n$, by Lemma $\ref{a priori}$. Further, $\eqref{truncation bound}$ implies that $\{T_k(u_n)\}$ is a uniformly bounded sequence in $X_0^{s,p,\alpha}(\Omega)$ for every $k>0$. Before proving the existence of an entropy solution, we provide the following compactness result.
\begin{lemma}[Lemma 3.6, \cite{Abdellaoui 1}]\label{strong}
	Let the sequence $\{u_n\}\subset X_0^{s,p,\alpha}(\Omega)$ be a positive increasing sequence w.r.t $n$ and $(-\Delta)^s_{p,\alpha}u_n\geq0$ in $\Omega$. Furthermore, assume the sequence $\{T_k(u_n)\}$ is uniformly bounded in $X_0^{s,p,\alpha}(\Omega)$ for every $k>0$. Then, there exists $u$ such that $u_n\uparrow u$ a.e. in $\Omega$, $T_k(u)\in X_0^{s,p,\alpha}(\Omega)$ and $$T_k(u_n)\rightarrow T_k(u) \text{ strongly in } X_0^{s,p,\alpha}(\Omega).$$
\end{lemma}
\noindent We now prove Theorem $\ref{exist entropy}$ and show the equivalence between the weak solution and the entropy solution to $\eqref{ent}$.
	\begin{proof}[Proof of Theorem $\ref{exist entropy}$.]
		The proof of this theorem is divided into the following two parts.\\
\textbf{Part 1:} There exists a positive entropy solution to $\eqref{ent}$ in the sense of Definition $\ref{entropy}$.\\	
By Theorem $\ref{exist weak}$, we guarantee the existence of a weak solution $u$ to $\eqref{ent}$. Then, $\eqref{truncation bound}$ and Lemma $\ref{strong}$ imply that $T_k(u_n)\rightarrow T_k(u)$ strongly in $X_0^{s,p,\alpha}(\Omega)$.	\\
Let us take $\phi=T_1(G_l(u_n))$, for a fixed $l>0$, in $\eqref{weak formula approx}$ and we get
\begin{align}\label{entropy 3}
 \int_{Q}|u_n(x)-u_n(y)|^{p-2}&(u_n(x)-u_n(y))(T_1(G_l(u_n))(x)-T_1(G_l(u_n))(y))d\mu\nonumber\\&=\int_{\Omega}f_nh_n(u_n+1/n)T_1(G_l(u_n))+\int_{\Omega}g_n T_1(G_l(u_n))\nonumber\\&\leq \int_{\{u_n\geq l\}}f_n h_n(l+1/n)+\int_{\{u_n\geq l\}}g_n.
  \end{align}
 Consider the set $D_l$ as given in Definition $\ref{entropy}$. Thus, for $(x,y)\in D_l$ it is easy to show that
  \begin{align}\label{D_l}
   |u_n(x)-u_n(y)|^{p-1}\leq  |u_n(x)-u_n(y)|^{p-2}&(u_n(x)-u_n(y))(T_1(G_l(u_n))(x)-T_1(G_l(u_n))(y)).
  \end{align}
By considering $\eqref{entropy 3}$, $\eqref{D_l}$ and using the Fatou's lemma, we have
  \begin{align}\label{entropy 4}
 \int_{D_l}|u(x)-u(y)|^{p-1} d\mu& \leq  \lim\limits_{n\rightarrow\infty}\inf  \left\{\int_{D_l}|u_n(x)-u_n(y)|^{p-1} d\mu\right\}\nonumber\\& \leq \int_{\{u\geq l\}}f h(l)+\int_{\{u\geq l\}}g.\nonumber
  \end{align}
 In fact,
 $$\int_{\{u\geq l\}}f h(l)+\int_{\{u\geq l\}}g\rightarrow 0 ~\text{ as }~ l\rightarrow \infty.$$ 
 Hence, we establish $\eqref{condition 1}$, i.e. we have
 $$ \int_{D_l}|u(x)-u(y)|^{p-1} d\mu\rightarrow 0~\text{ as }~ l\rightarrow \infty.$$
 Now it remains to prove $\eqref{condition 2}$. For this, we consider  $\varphi\in X_0^{s,p,\alpha}(\Omega)\cap L^\infty(\Omega)$. Using $T_k(u_n-\varphi)$ as a test function in $\eqref{weak formula approx}$ and following the notations used in the proof of Theorem $\ref{exist weak}$, we get
\begin{equation}\label{entropy 5}
\int_{Q}I_n(x,y)(T_k(u_n-\varphi)(x)-T_k(u_n-\varphi)(y))d\mu
=\int_{\Omega}f_nh_n(u_n+1/n)T_k(u_n-\varphi)+\int_{\Omega}g_n T_k(u_n-\varphi).
\end{equation}
The integrand in the first term of the above equation $\eqref{entropy 5}$ can be decomposed as
\begin{equation}\label{entropy 6}
I_n(x,y)(T_k(u_n-\varphi)(x)-T_k(u_n-\varphi)(y))= I_{1,n}(x,y)+I_{2,n}(x,y),
\end{equation}
where
\begin{align}
I_{1,n}(x,y)=|(u_n(x)-\varphi(x))-(u_n(y)-\varphi(y))|^{p-2}&\left((u_n(x)-\varphi(x))-(u_n(y)-\varphi(y))\right)\nonumber\\&\times \left[T_k(u_n-\varphi)(x)-T_k(u_n-\varphi)(y)\right]\nonumber
\end{align} and 
\begin{align}
I_{2,n}(x,y)=&\left[I_n(x,y)-|(u_n(x)-\varphi(x))-(u_n(y)-\varphi(y))|^{p-2}\left((u_n(x)-\varphi(x))-(u_n(y)-\varphi(y))\right)\right]\nonumber\\&~~~~\times \left[T_k(u_n-\varphi)(x)-T_k(u_n-\varphi)(y)\right].\nonumber
\end{align}
 Clearly, $I_{1,n}(x,y)\geq 0$ a.e. in $Q$ and 
 \begin{align}
 I_{1,n}(x,y)\rightarrow |(u(x)-\varphi(x))-(u(y)-\varphi(y))|^{p-2}&\left((u(x)-\varphi(x))-(u(y)-\varphi(y))\right)\nonumber\\&\times \left[T_k(u-\varphi)(x)-T_k(u-\varphi)(y)\right] ~\text{ a.e. in}~Q.\nonumber
 \end{align}
 Thus, on using the Fatou's lemma we establish
 \begin{align}\label{entropy 7}
\lim\limits_{n\rightarrow \infty}\inf\left\{\int_{Q} I_{1,n}(x,y)d\mu\right\}\geq&\int_{Q} |(u(x)-\varphi(x))-(u(y)-\varphi(y))|^{p-2}\nonumber\\&~\left((u(x)-\varphi(x))-(u(y)-\varphi(y))\right)\times \left[T_k(u-\varphi)(x)-T_k(u-\varphi)(y)\right] d\mu.
 \end{align}
 According to Abdellaoui et al. \cite{Abdellaoui 1}, 
\begin{align}\label{entropy 8}
\lim\limits_{n\rightarrow \infty}&\int_{Q} I_{2,n}(x,y)d\mu\nonumber\\&= \int_{Q}\left[I(x,y)-|(u(x)-\varphi(x))-(u(y)-\varphi(y))|^{p-2}\left((u(x)-\varphi(x))-(u(y)-\varphi(y))\right)\right]\nonumber\\&~~~~~\times \left[T_k(u-\varphi)(x)-T_k(u-\varphi)(y)\right] d\mu.
\end{align}
On the other hand, by using the dominated convergence theorem, Lemma $\ref{a priori}$ and the strong convergence $T_k(u_n)\rightarrow T_k(u)$ in $X_0^{s,p,\alpha}(\Omega)$, we observe that
\begin{equation}\label{entropy 9}
\int_{\Omega}f_nh_n(u_n+1/n)T_k(u_n-\varphi)+\int_{\Omega}g T_k(u_n-\varphi)\rightarrow\int_{\Omega}fh(u)T_k(u-\varphi)+\int_{\Omega}g T_k(u-\varphi),~\text{as}~n\rightarrow\infty.
\end{equation}
Therefore, on combining the results from $\eqref{entropy 6}-\eqref{entropy 9}$ and then passing the limit $n\rightarrow\infty$ in $\eqref{entropy 5}$, we conclude
\begin{equation}
\int_{Q}I(x,y)(T_k(u-\varphi)(x)-T_k(u-\varphi)(y))d\mu
\leq\int_{\Omega}fh(u)T_k(u-\varphi)+\int_{\Omega} g T_k(u-\varphi),
\end{equation}
for every $\varphi\in X_0^{s,p,\alpha}(\Omega)\cap L^\infty(\Omega)$. Thus, $\eqref{condition 2}$ holds.\\
\textbf{Part 2:} Uniqueness of entropy solutions.\\
Let $u$ be the positive entropy solution obtained from Part 1 of this proof and $u_n$ be the unique solution to $\eqref{ent approx}$. Then, from Lemma $\ref{strong}$, $u=\lim\limits_{n\rightarrow \infty}\sup \{u_n\}$. \\
We prove this theorem by method of contradiction. For that, suppose $\bar{u}$ is another entropy solution to $\eqref{ent}$. Let us fix $n$ and define $\psi_n=(u_n-\bar{u})^+$. Thus, for $k\gg \|u_n\|_{L^{\infty}(\Omega)}$, $\psi_n=(u_n-T_k(\bar{u}))^+$. This implies $\psi_n\in X_0^{s,p,\alpha}(\Omega)\cap L^\infty(\Omega)$ and $\psi_n\equiv 0$ in the set $\{\bar{u}>\|u_n\|_{L^{\infty}(\Omega)}\}$. Hence, on choosing $\phi=\psi_n$ in $\eqref{weak formula approx}$, we have
\begin{align}\label{entropy 10}
\int_{Q}|u_n(x)-u_n(y)|^{p-2}(u_n(x)-u_n(y))(\psi_n(x)-\psi_n(y))d\mu&=\int_{\Omega}f_nh_n(u_n+1/n)\psi_n+\int_{\Omega}g_n \psi_n\nonumber\\&\leq \int_{\Omega}f h(u_n)\psi_n+\int_{\Omega}g \psi_n.
\end{align}
Since $\bar{u}$ is an entropy solution to $\eqref{ent}$, by Lemma $\ref{another weak form}$, we get
\begin{align}\label{entropy 11}
\int_{Q}|\bar{u}(x)-\bar{u}(y)|^{p-2}(\bar{u}(x)-\bar{u}(y))(\psi_n(x)-\psi_n(y))d\mu& = \int_{\Omega}f h(\bar{u})\psi_n+\int_{\Omega}g \psi_n.
\end{align}
Subtracting $\eqref{entropy 11}$ from $\eqref{entropy 10}$ we reach that
\begin{align}
0 &\leq C \int_{Q}|\psi_n(x)-\psi_n(y)|^p\nonumber\\&\leq \int_{Q}\left(|u_n(x)-u_n(y)|^{p-2}(u_n(x)-u_n(y))-|\bar{u}(x)-\bar{u}(y)|^{p-2}(\bar{u}(x)-\bar{u}(y))\right)(\psi_n(x)-\psi_n(y))d\mu\nonumber\\&\leq \int_\Omega f(h(u_n)-h(\bar{u}))\psi_n\nonumber\\&\leq 0.\nonumber
\end{align}
Therefore, $\psi_n=0$ a.e. in $Q$ and $u_n\leq \bar{u}$ a.e. for every $n$. Use the fact $u=\lim\limits_{n\rightarrow \infty}\sup \{u_n\}$ to obtain
\begin{equation}\label{one side}
u\leq \bar{u} ~\text{a.e in}~Q.
\end{equation}
Further, choose $h\gg k$. Then, the two entropy solutions $u$ and $\bar{u}$ satisfy
\begin{equation}\label{entropy 12}
\int_{Q}I(x,y)[T_k(u(x)-T_h(\bar{u}(x)))-T_k(u_n(y)-T_h(\bar{u}(y)))]d\mu
\leq\int_{\Omega}(fh(u)+g)T_k(u-T_h(\bar{u}))
\end{equation} and
\begin{equation}\label{entropy 13}
\int_{Q}\bar{I}(x,y)[T_k(\bar{u}(x)-T_h(u(x)))-T_k(\bar{u}(y)-T_h(u(y)))]d\mu
\leq\int_{\Omega}(fh(\bar{u})+g)T_k(\bar{u}-T_h(u)),
\end{equation}
where $$\bar{I}(x,y)=|\bar{u}(x)-\bar{u}(y)|^{p-2}(\bar{u}(x)-\bar{u}(y))~\text{ and }~I(x,y)=|u(x)-u(y)|^{p-2}(u(x)-u(y)).$$
Clearly, as $h\rightarrow\infty$,
\begin{equation}\label{entropy 14}
\int_{\Omega}gT_k(u-T_h(\bar{u}))+\int_{\Omega}gT_k(\bar{u}-T_h(u))\rightarrow 0 
\end{equation}
and 
\begin{equation}\label{entropy 15}
\int_{\Omega}fh(u)T_k(u-T_h(\bar{u}))+\int_{\Omega}fh(\bar{u})T_k(\bar{u}-T_h(u))\rightarrow \int_{\Omega}fh(u)T_k(u-\bar{u})+\int_{\Omega}fh(\bar{u})T_k(\bar{u}-u).
\end{equation}
From $\eqref{one side}$, we already have $u\leq \bar{u}$ a.e. in $Q$. Thus, the right hand side of $\eqref{entropy 15}$ becomes
\begin{align}\label{entropy 16}
\int_{\Omega}fh(u)T_k(u-\bar{u})+\int_{\Omega}fh(\bar{u})T_k(\bar{u}-u)&= \int_{\Omega}f(h(u)-h(\bar{u}))T_k(u-\bar{u})\nonumber\\&\leq 0.
\end{align}
Let us define the following two sets:
$$D(h)=\{(x,y)\in Q:u(x)<h,u(y)<h\} \text{ and }\bar{D}(h)=\{(x,y)\in Q:\bar{u}(x)<h,\bar{u}(y)<h\}.$$
Since $u\leq \bar{u}$, we get $u<h$ in the set $\{\bar{u}<h\}$. On combining the inequalities $\eqref{entropy 12}-\eqref{entropy 16}$, we obtain 
\begin{align}\label{entropy 17}
o(h)&\geq \int_{Q}I(x,y)(T_k(u(x)-T_h(\bar{u}(x)))-T_k(u_n(y)-T_h(\bar{u}(y))))d\mu\nonumber\\&~~~~+\int_{Q}\bar{I}(x,y)(T_k(\bar{u}(x)-T_h(u(x)))-T_k(\bar{u}(y)-T_h(u(y))))d\mu\nonumber\\&= \int_{D(h)}I(x,y)(T_k(u(x)-T_h(\bar{u}(x)))-T_k(u_n(y)-T_h(\bar{u}(y))))d\mu\nonumber\\&~~~~+\int_{Q\setminus D(h)}I(x,y)(T_k(u(x)-T_h(\bar{u}(x)))-T_k(u_n(y)-T_h(\bar{u}(y))))d\mu\nonumber\\&~~~~+\int_{\bar{D}(h)}\bar{I}(x,y)(T_k(\bar{u}(x)-u(x))-T_k(\bar{u}(y)-u(y)))d\mu\nonumber\\&~~~~+\int_{Q\setminus \bar{D}(h)}\bar{I}(x,y)(T_k(\bar{u}(x)-T_h(u(x)))-T_k(\bar{u}(y)-T_h(u(y))))d\mu\nonumber
\end{align}
\begin{align}
&=\int_{\bar{D}(h)}[\bar{I}(x,y)-I(x,y)](T_k(\bar{u}(x)-u(x))-T_k(\bar{u}(y)-u(y)))d\mu\nonumber\\&~~~~+\int_{D(h)\setminus \bar{D}(h)}I(x,y)(T_k(u(x)-T_h(\bar{u}(x)))-T_k(u_n(y)-T_h(\bar{u}(y))))d\mu\nonumber\\&~~~~+\int_{Q\setminus D(h)}I(x,y)(T_k(u(x)-h)-T_k(u_n(y)-h))d\mu\nonumber\\&~~~~+ \int_{Q\setminus \bar{D}(h)}\bar{I}(x,y)(T_k(\bar{u}(x)-T_h(u(x)))-T_k(\bar{u}(y)-T_h(u(y))))d\mu\nonumber\\&= J_1(h)+J_2(h)+J_3(h)+J_4(h).
\end{align}
It is easy to check that 
\begin{align}\label{entropy 18}
J_1(h)&=\int_{\bar{D}(h)}[\bar{I}(x,y)-I(x,y)](T_k(\bar{u}(x)-u(x))-T_k(\bar{u}(y)-u(y)))d\mu\nonumber\\&\geq C \int_{\bar{D}(h)} |T_k(\bar{u}(x)-u(x))-T_k(\bar{u}(y)-u(y))|^p d\mu.
\end{align} 
According to Abdellaoui et al. \cite{Abdellaoui 1}, 
\begin{equation}\label{entropy 19}
J_2(h)+J_3(h)+J_4(h)\geq o(h).
\end{equation}
Hence, by the inequalities $\eqref{entropy 17}-\eqref{entropy 19}$, we establish the following.
$$ o(h)+C \int_{\bar{D}(h)} |T_k(\bar{u}(x)-u(x))-T_k(\bar{u}(y)-u(y))|^p d\mu\leq o(h).$$
Now on passing the limit $h\rightarrow \infty$, we conclude that  
$$\int_{\mathbb{R}^N} |T_k(\bar{u}(x)-u(x))-T_k(\bar{u}(y)-u(y))|^p d\mu=0.$$
This implies $T_k(\bar{u}-u)$ is a constant function. Since $u=\bar{u}=0$ in $\mathbb{R}^N\setminus\Omega$, we reach at the conclusion that $u=\bar{u}$ a.e. in $Q$.\\
Thus, from Part 1 and Part 2 of this proof we conclude that $u\in \mathcal{T}^{s,p,\alpha}_0(\Omega)$ is a unique positive entropy solution to $\eqref{ent}$ in the sense of Definition $\ref{entropy}$ and is equivalent to the weak solution as obtained in Theorem $\ref{exist weak}$.
	\end{proof}
	\subsection{Problem $\eqref{ent}$ with a power nonlinearity.}
\noindent We consider the following problem with a singularity, power nonlinearity and an $L^1$ datum. We show that this problem possesses an entropy solution. 
	\begin{equation}\label{ent special}
\begin{split}
(-\Delta)_{p,\alpha}^sw&= \frac{f(x)}{w^\gamma}+w^r+g(x) ~\text{in}~\Omega,~r<p-1,\gamma<1,\\
w&>0~\text{in}~\Omega,\\
w&= 0~\text{in}~\mathbb{R}^N\setminus\Omega.
\end{split}
\end{equation}
The following is the definition of entropy solution to the problem $\eqref{ent special}$.
	\begin{definition}\label{entropy special}
	 A function $w\in \mathcal{T}^{s,p,\alpha}_0(\Omega)$ is said to be an entropy solution of $\eqref{ent special}$ if it satisfies $\eqref{condition 1}$ and 
for every $\varphi\in W^{s,p,\alpha}_0(\Omega)\cap L^\infty(\Omega)$ and every $k>0$,
	\begin{align}\label{condition}
	\int_{Q}|w(x)-w(y)|^{p-2}(w(x)-w(y))&[T_k(w(x)-\varphi(x))-T_k(w(y)-\varphi(y))]d\mu\nonumber\\&\leq\int_{\Omega}\frac{f}{w^\gamma}T_k(w-\varphi)+\int_{\Omega}w^r T_k(w-\varphi)+\int_{\Omega}g T_k(w-\varphi).
	\end{align}
	Further, for every $\omega\subset \subset\Omega$, there exists a $C_\omega>0$ such that $w\geq C_\omega>0$.	
\end{definition}
\begin{theorem}
	Let $r<p-1$ and $0<\gamma<1$. Then, the problem $\eqref{ent special}$ admits a positive entropy solution in $\mathcal{T}_0^{s,p,\alpha}(\Omega)$ in the sense of Definition $\ref{entropy special}$.
\end{theorem}
\begin{proof}
Denote $f_n=T_n(f)$ and let $\{g_n\}\subset L^\infty(\Omega)$ be a nonnegative, increasing sequence converging strongly to $g$ in $L^1(\Omega)$. Consider the following approximating problem: 
\begin{equation}\label{ent special approx}
\begin{split}
(-\Delta)_{p,\alpha}^sw_n&= \frac{f_n(x)}{(w_n+1/n)^\gamma}+w_n^r+g_n(x) ~\text{in}~\Omega,\\
w_n&>0~\text{in}~\Omega,\\
w_n&= 0~\text{in}~\mathbb{R}^N\setminus\Omega.
\end{split}
\end{equation}
The corresponding energy functional is given by 
\begin{align}
\Psi(w_n)&=\frac{1}{p}\int_{Q}|w_n(x)-w_n(y)|^p d\mu-\frac{1}{1-\gamma}\int_{\Omega}f_n\left((w_n^++1/n)^{1-\gamma}-(1/n)^{1-\gamma}\right)\nonumber\\&~~~~-\frac{1}{r+1}\int_{\Omega}(w_n^+)^{r+1}-\int_\Omega g_n w_n^+.
\end{align}
The functional $\Psi$ is $C^1$, coercive and weakly lower semi-continuous in $X_0^{s,p,\alpha}(\Omega)$. Thus, there exists a critical point of $\Psi$ and hence a weak solution  to the problem $\eqref{ent special}$, denoted by $w_n\in X_0^{s,p,\alpha}(\Omega)$. The strong maximum principle [Lemma 2.3, \cite{Mosconi}] and the comparison principle, i.e. Lemma $\ref{comparison}$, imply that the solution $w_n>0$ is unique and the sequence $(w_n)$ is an increasing sequence with respect to $n$.\\
Clearly, by Lemma $\ref{comparison}$, we obtain $u_n\leq w_n$ in $\Omega$ for every $n$, where $u_n$ is the positive weak solution to $\eqref{ent approx}$ with $h(t)=t^{-\gamma}$. According to Lemma $\ref{a priori}$, $u_n^{-\gamma}f\in L^1(\Omega)$ and hence $w_n^{-\gamma}f\in L^1(\Omega)$. \\
\textbf{Claim:} the sequence $\{w_n\}$ is uniformly bounded in $L^{p-1}(\Omega)$.\\
We will use the method of contradiction to prove our claim. Assume that $\{w_n\}$ is an unbounded sequence in $L^{p-1}(\Omega)$. Hence, there exists a subsequence of $\{w_n\}$, still denoted as $\{w_n\}$, such that $M_n=\|w_n^{p-1}\|_{L^1(\Omega)}\rightarrow\infty$ as $n\rightarrow\infty$. Let us define 
$$\bar{w}_n=\frac{w_n}{M_n^{\frac{1}{p-1}}}.$$
This implies $\|\bar{w}_n\|_{L^{p-1}(\Omega)}=1.$ Since $w_n$ is a solution to $\eqref{ent special approx}$, $\bar{w}_n$ satisfies 
\begin{equation}\label{contradiction }
\begin{split}
(-\Delta)_{p,\alpha}^s\bar{w}_n&=  \frac{f_n(x)}{M_n({w_n}+1/n)^\gamma}+M_n^{\frac{r-p+1}{p-1}}\bar{w}_n^r+M_n^{-1}g_n(x) ~\text{in}~\Omega,\\
\bar{w}_n&>0~\text{in}~\Omega,\\
\bar{w}_n&= 0~\text{in}~\mathbb{R}^N\setminus\Omega.
\end{split}
\end{equation}
Denote $H_n= \frac{f_n}{M_n({w_n}+1/n)^\gamma}+M_n^{\frac{r-p+1}{p-1}}\bar{w}_n^r+M_n^{-1}g_n$. Then, it is simple to prove that $\|H_n\|_{L^1(\Omega)}\rightarrow 0$ as $n\rightarrow\infty$. Following the proof of Lemma $\ref{bounded}$, we establish that $\{\bar{w}_n\}$ is bounded in $L^{m_1}(\Omega,d\nu)$ for every $m_1<\frac{N(p-1)}{N-sp}$ and also bounded in $X_0^{t,m,\alpha}(\Omega)$ for every $1\leq m<\frac{N(p-1)}{N-s}$, $0<t<s$. Thus, there exists $\bar{w}\in X_0^{t,m,\alpha}(\Omega)\cap L^{m_1}(\Omega, d\nu)$ such that $T_k(\bar{w})\in X_0^{s,p,\alpha}(\Omega) $  and $T_k(\bar{w}_n)\rightarrow T_k(\bar{w})$ weakly in $X_0^{s,p,\alpha}(\Omega)$. We now use the Vitali's lemma to prove $\bar{w}_n^{p-1}\rightarrow \bar{w}^{p-1}$ strongly in $L^1(\Omega)$ and we obtain $\|\bar{w}\|_{L^{p-1}}(\Omega)=1$.\\
On readopting the proof of Lemma $\ref{bounded}$ with the test function $T_k(\bar{w}_n)$ we obtain $$\|T_k(\bar{w}_n)\|_{X_0^{s,p,\alpha}(\Omega)}\rightarrow 0,~\text{as}~n\rightarrow\infty.$$
Thus, $T_k(\bar{w})=0$ for every $k$ and this contradicts to $\|\bar{w}\|_{L^{p-1}(\Omega)}=1$. This proves the claim and $\{w_n\}$ is uniformly bounded in $L^r(\Omega)$.\\
We can now guarantee the existence of a weak solution $w\in X_0^{t,m,\alpha}(\Omega)$ to the problem $\eqref{ent special}$ for every $1\leq m<\frac{N(p-1)}{N-s}$, $0<t<s$. Further, $T_k({w_n})\rightarrow T_k({w})$ weakly in $X_0^{s,p,\alpha}(\Omega)$. The proof follows verbation of the proofs in Lemma $\ref{bounded}$ and Theorem $\ref{exist weak}$.\\
We have already shown that $\{w_n\}$ is an increasing sequence. Hence, by Lemma $\ref{strong}$, $T_k({w_n})$ converges strongly to $T_k({w})$ in $X_0^{s,p,\alpha}(\Omega)$. Proceeding on the similar lines used in the proof of Theorem $\ref{exist entropy}$, we get the existence of an entropy solution $w$ to $\eqref{ent special}$ in the sense of Definition $\ref{entropy special}$. Hence the proof.
\end{proof}
	\section*{Acknowledgement}
 The author Akasmika Panda thanks the financial assistantship received from the Ministry of Human Resource Development (M.H.R.D.), Govt. of India. The author Debajyoti Choudhuri thanks SERB for the financial assistanceship received to carry out this research vide MTR/2018/000525. Both the authors also acknowledge the facilities received from the Department of Mathematics, National Institute of Technology Rourkela.
 \section*{Data Availability statement}
 \noindent The article has no data associated to it whatsoever.
 \section*{Conflict of interest statement}
 \noindent There is no conflict of interest whatsoever.
 
\end{document}